\documentclass[11pt, a4paper, oneside, leqno]{article}

\usepackage[utf8]{inputenc}

\usepackage{amsthm}
\usepackage{amsmath}
\usepackage{amssymb}

\usepackage[square,comma]{natbib}

\usepackage[%
    bookmarks,
    colorlinks,
    breaklinks,
    linkcolor=blue,
    citecolor=blue,
    filecolor=blue,
    urlcolor=blue
]{hyperref}

\usepackage{doi}

\usepackage{braket}

\newtheorem{thm}{Theorem}[section]
\theoremstyle{definition}
\newtheorem{dfn}[thm]{Definition}
\theoremstyle{plain}

\newcommand{\Tleq}{\leq_{\mathbf{T}}}
\newcommand{\Tlneq}{\lneq_{\mathbf{T}}}

\newcommand{\Nat}[0]{\mathbb{N}}

\newcommand{\PowN}[0]{\mathcal{P}(\mathbb{N})}

\newcommand{\lt}[0]{\mathrm{L}_2}
\newcommand{\z}[0]{\mathsf{Z}}

\newcommand{\rca}[0]{\mathsf{RCA}}
\newcommand{\wwkl}[0]{\mathsf{WWKL}}
\newcommand{\wkl}[0]{\mathsf{WKL}}
\newcommand{\rt}[0]{\mathsf{RT}}
\newcommand{\aca}[0]{\mathsf{ACA}}
\newcommand{\atr}[0]{\mathsf{ATR}}
\newcommand{\dica}[0]{\Delta^1_1\text{-}\mathsf{CA}}
\newcommand{\pica}[0]{\Pi^1_1\text{-}\mathsf{CA}}

\newcommand{\pa}[0]{\mathrm{PA}}

\title{Set existence principles and closure conditions:
unravelling the standard view of reverse mathematics}
\author{Benedict Eastaugh%
\footnote{%
Munich Center for Mathematical Philosophy,
LMU Munich,
Geschwister-Scholl-Platz 1,
80539 Munich,
Germany.
Email: \href{mailto:benedict@eastaugh.net}{benedict@eastaugh.net}
}}

\begin{document}

\maketitle

\begin{abstract}
    \noindent%
    It is a striking fact from reverse mathematics that almost all theorems of
    countable and countably representable mathematics are equivalent to just
    five subsystems of second order arithmetic.
    The standard view is that the significance of these equivalences lies in
    the set existence principles that are necessary and sufficient to prove
    those theorems.
    In this article I analyse the role of set existence principles in reverse
    mathematics, and argue that they are best understood as closure conditions
    on the powerset of the natural numbers.
\end{abstract}

\section{Introduction}
\label{sec:introduction}

What axioms are truly necessary for proving particular mathematical theorems?
To answer this question, Harvey Friedman \citeyearpar{Friedman1975,
Friedman1976} initiated a research program called reverse mathematics. By
formalising ordinary mathematical concepts in the language of second order
arithmetic, Friedman was able to show not only that many theorems of ordinary
mathematics could be proved in relatively weak subsystems of second order
arithmetic, but that such theorems often turned out to be equivalent to the
axioms used to prove them, with this equivalence being provable in a weak base
theory.
Consider, for example, the least upper bound axiom for the real numbers. Not
only are different formulations of the least upper bound axiom provably
equivalent, but they are also provably equivalent to the axiom scheme of
arithmetical comprehension. Reverse mathematics thus provides a unified
framework within which one can precisely determine the strength of the axioms
necessary to prove theorems of ordinary mathematics.

Reverse mathematics is carried out within subsystems of second order arithmetic,
for which the canonical reference is \citet{Simpson2009}.
These subsystems are formulated in the \emph{language of second order
arithmetic} $\lt$. This is a two-sorted first order language, the first sort of
variables $x_0, x_1, \dotsc$ being called \emph{number variables}, and intended
to range over natural numbers, and the second sort of variables
$X_0, X_1, \dotsc$ being called \emph{set variables} and intended to range over
sets of natural numbers. $\lt$ includes the non-logical symbols of first order
Peano arithmetic (constant symbols $0$ and $1$, function symbols $+$ and
$\times$, and the relation symbol $<$), as well as a relation symbol $\in$
standing for set membership.

The system of \emph{full second order arithmetic} or $\z_2$ consists of: the
\emph{basic axioms} of Peano arithmetic minus the induction scheme, $\pa^-$;
the \emph{$\Sigma^0_1$ induction scheme}, consisting of the universal closures
of all formulas of the form
\begin{equation*}
    (\varphi(0) \wedge \forall{n} (\varphi(n) \rightarrow \varphi(n + 1)))
    \rightarrow
    \forall{n} \; \varphi(n)
    \label{eqn:sigma_0_1_ind}
    \tag{$\Sigma^0_1\text{-}\mathsf{IND}$}
\end{equation*}
where $\varphi(n)$ is a $\Sigma^0_1$ formula (which may contain free number and
set variables);
and the \emph{second order comprehension scheme}, which consists of the
universal closures of all formulas of the form
\begin{equation*}
    \exists{X} \forall{n} (n \in X \leftrightarrow \varphi(n))
    \label{eqn:full_ca}
    \tag{$\Pi^1_\infty\text{-}\mathsf{CA}$}
\end{equation*}
where $\varphi(n)$ is an $\lt$-formula with $X$ not free (but possibly with
other free number and set variables). $\z_2$ is, in the context of arithmetic, a
very strong system, with consistency strength far beyond that of first order
Peano arithmetic $\pa$.
The systems employed in reverse mathematics are all subsystems of $\z_2$, in the
sense that all of their axioms are theorems of $\z_2$.%
\footnote{%
In order to achieve a streamlined and uniform presentation of $\z_2$ and its
subsystems, this article deviates slightly from \citet{Simpson2009}, mainly in
its treatment of induction axioms. This does not affect the results quoted in
any material way.
}

The most fundamental subsystem of second order arithmetic for reverse
mathematics is $\rca_0$. Its axioms are: the basic axioms; the $\Sigma^0_1$
induction scheme; and the \emph{recursive} (or $\Delta^0_1$) \emph{comprehension
scheme}, which consists of the universal closures of all formulas of the form
\begin{equation}
    \forall{n} (\varphi(n) \leftrightarrow \psi(n))
    \rightarrow
    \exists{X} \forall{n} (n \in X \leftrightarrow \varphi(n))
    \label{eqn:delta_0_1_ca}
    \tag{$\Delta^0_1\text{-}\mathsf{CA}$}
\end{equation}
where $\varphi$ is a $\Sigma^0_1$ formula, $\psi$ is a $\Pi^0_1$ formula, and
$X$ is not free in $\varphi$ (although $\varphi$ may contain other free
variables). It is the Recursive Comprehension Axiom that gives $\rca_0$ its
name.%
\footnote{%
The `$0$' subscript indicates that this is a system with a restricted induction
axiom, i.e. it does not prove the full second order induction scheme.
}
$\rca_0$ is the standard base theory for reverse mathematics: the system in
which the equivalences between theorems and axioms are typically proved. In the
rest of this article, unqualified expressions of the form ``$\tau$ is equivalent
to $S$'', where $\tau$ is a mathematical theorem and $S$ a subsystem of second
order arithmetic, should be taken to mean that a faithful formalisation of
$\tau$ can be proved in $\rca_0$ to be implied by, and imply, the axioms of $S$.
Similarly, ``$\tau$ reverses to $S$'' means that a faithful formalisation of
$\tau$ can be proved in $\rca_0$ to imply the axioms of $S$.%
\footnote{%
The term ``reversal'' implies that the implication $S \Rightarrow \tau$ is
already known, so I will not characterise implications $\tau \Rightarrow T$
where $T$ is proof-theoretically weaker than $\tau$ as ``reversals''; instead I
shall simply call them ``implications''.
}
As $\rca_0$ is the base theory for most reverse mathematics, the other systems
considered in this article generally include the axioms of $\rca_0$. The system
$\wkl_0$, for example, consists of the axioms of $\rca_0$ plus the axiom known
as \emph{weak König's lemma}, which states that every infinite subtree of the
full binary tree $2^{<\Nat}$ has an infinite path through it. In such cases we
say that the \emph{distinguishing axiom} of $\wkl_0$ is weak König's lemma.

While there are many other subsystems of second order arithmetic, a particularly
important group are known as the Big Five. As well as $\rca_0$ and $\wkl_0$,
this group consists of $\aca_0$, whose distinguishing axiom is the
\emph{arithmetical comprehension scheme}; $\atr_0$, whose distinguishing axiom
is the scheme of \emph{arithmetical transfinite recursion}; and $\pica_0$, whose
distinguishing axiom is the \emph{$\Pi^1_1$-comprehension scheme}.%
\footnote{%
The historical development of these canonical five subsystems of second order
arithmetic is traced in \citet{DeaWal2016}.
}
Each of these systems is stronger than the preceding systems, in the following
sense: a system $S_2$ is \emph{proof-theoretically stronger} than a system
$S_1$, $S_1 < S_2$, just in case $S_2$ proves all the theorems of $S_1$, but
there is at least one theorem of $S_2$ that is not a theorem of $S_1$. For the
Big Five we have that $\rca_0 < \wkl_0 < \aca_0 < \atr_0 < \pica_0$, although it
should be noted that increases in proof-theoretic strength do not necessarily
involve increases in consistency strength.

In most cases studied to date, ordinary mathematical theorems concerning
countable and countably-representable objects have been found to be either
provable in the base theory $\rca_0$ or are equivalent over $\rca_0$ to another
of the Big Five.
This is a remarkable phenomenon: \citet[p.~115]{Simpson2010} estimates that
hundreds of theorems have been found that fall into these five equivalence
classes.
There is also is a substantial and growing body of statements that fall outside
of this classification, known as the Reverse Mathematics Zoo
\citep{Dzhafarov2015}. These were originally drawn largely from infinitary
combinatorics, by considering weakenings of Ramsey's theorem; an excellent
starting point for the study of these statements is \citet{Hirschfeldt2014}.
However, the Zoo now contains examples from many other areas of mathematics
including model theory, set theory, the theory of linear orderings, and
descriptive set theory.
Relatively few theorems from other core areas of mathematics such as analysis or
algebra are known to lie outside the Big Five, although this situation is also
changing: a recent example from analysis is Birkhoff's recurrence theorem
\citep{Day2016}.

Explaining why the Big Five phenomenon occurs is an important task, and if an
account of the significance of reversals could accomplish it, this would be a
substantial accomplishment. However, it is also not one I take to be necessary
for such accounts, and consequently I shall not address it directly in this
article.
One explanation offered by \citet{Montalban2011} is that the Big Five, and
perhaps some other subsystems, are robust systems: they are invariant under
certain perturbations of the axioms involved.
Elaborations on this view can be found in \citet{Sanders2012, Sanders2014}.

The standard view in the field of reverse mathematics is that the significance
of reversals lies in their ability to demonstrate what set existence principles
are necessary to prove theorems of ordinary mathematics. For example, they show
that the arithmetical comprehension scheme is necessary in order to prove the
least upper bound axiom for the real numbers, because the latter implies the
former over the weak base theory $\rca_0$. In contrast, arithmetical
comprehension is not necessary in order to prove the Hahn--Banach theorem for
separable Banach spaces, since this theorem can be proved in a weaker system
than $\aca_0$, namely $\wkl_0$. However, weak König's lemma is necessary to
prove the separable Hahn--Banach theorem, and this necessity is demonstrated by
the fact that $\rca_0$ proves that the latter implies the former.
A view of this sort, in more or less the terms just used, is articulated by
\citet[p.~2]{Simpson2009} as his ``Main Question'':
``Which set existence axioms are needed to prove the theorems of ordinary,
non-set-theoretic mathematics?''.
Similar sentiments can be found elsewhere.%
\footnote{%
Such as in
\citet*[p.~141]{FriSimSmi1983},
\citet[p.~557]{BroSim1986},
\citet[p.~172]{Jae1987},
\citet*[p.~191]{BroGiuSim2002},
\citet[p.~139]{AviSim2006} and
\citet*[p.~2]{DorDzhHir2015}.
There are many more examples to be found in the reverse mathematics literature,
although it should be noted that many of the participants are students or
coauthors of Simpson, and thus the similarity in language is not surprising.
}

It is not difficult to see why the standard view became and remains the
orthodoxy, since in some sense the view can be simply read off from the
mathematical results: each of the distinguishing axioms of the Big Five assert
the existence of certain sets of natural numbers, and the hierarchy of proof-%
theoretic strength that we see in the Big Five can thus be understood as a
hierarchy of set existence principles of increasing power. It also provides
precise explanations of commonly expressed opinions such as ``Gödel's
completeness theorem is nonconstructive'': the completeness theorem reverses to
$\wkl_0$, which implies the existence of noncomputable sets, so any proof of
the completeness theorem will necessarily employ a nonconstructive (in the
sense of noncomputable) set existence principle.
Finally, it allows us to comprehend within a common framework the various
foundations for mathematics that can be faithfully formalised by subsystems of
second order arithmetic as being differentiated (in terms of their mathematical
consequences, rather than their philosophical justifications) by their
commitment to set existence principles of differing strengths.
\citet{Simpson2010} proposes that each of the Big Five corresponds to a
particular foundational program:
$\rca_0$ to computable mathematics;
$\wkl_0$ to a partial realisation of Hilbert's program, as in
\citet{Simpson1988};
$\aca_0$ to the predicative theory proposed by \citet{Weyl1918};
$\atr_0$ to the predicative reductionist approach pioneered by
\citet{Feferman1964};
and $\pica_0$ to the program of impredicative analysis found in
\citet{BucFefPohSie1981}.%
\footnote{%
Further references for these connections can be found in \citet{Simpson2010},
while a more thoroughgoing analysis of them appears in \citet{DeaWal2016}.
}

From the usual statement of the standard view, one might easily conclude that no
ordinary mathematical theorems have been found that are not equivalent to
particular set existence principles, but this is not the case.
A formal version of Hilbert's basis theorem was shown by \citet{Simpson1988a} to
be equivalent to the wellordering of a recursive presentation of the ordinal
$\omega^\omega$, while \citet{RatWei1993} showed that Kruskal's theorem is
equivalent to the wellordering of a recursive presentation of
$\vartheta\Omega^\omega$, the small Veblen ordinal. These wellordering
statements are $\Pi^1_1$, and prima facie are not set existence principles. If
anything, they are set nonexistence principles, stating that there exist no
infinite sets witnessing the illfoundedness of the recursive linear orders
representing the ordinals $\omega^\omega$ and $\vartheta\Omega^\omega$
respectively.
Another class of examples is provided by the fragments of the first order
induction scheme given by the $\Sigma^0_n$ induction and bounding schemes.
\citet{Hir1987} showed that Ramsey's theorem for singletons,
$\rt^1_{<\infty}$, is equivalent over $\rca_0$ to the $\Sigma^0_2$ bounding
scheme $\mathsf{B}\Sigma^0_2$.
Equivalences to stronger and weaker fragments of first order induction are
also proved in \citet{SimSmi1986} and \citet*{FriSimYu1993}.
These cases demonstrate that the standard view should not be interpreted as the
claim that all equivalences proved in reverse mathematics concern set existence
principles. Rather, in cases where a reversal from theorem $\tau$ to a subsystem
$S$ of second order arithmetic show that a set existence principle is required
to prove $\tau$, the reversal demonstrates which set existence principle is
required.

Since the base theory $\rca_0$ is both expressively impoverished and
proof-theoretically weak, the coding involved in representing ordinary
mathematical objects (such as real numbers, complete separable metric spaces,
Borel and analytic sets, and so on) in that theory is substantial. For reversals
to have the significance ascribed to them by the standard view requires that the
process of formalisation is faithful, in the sense that the formal statements in
the language of second order arithmetic formally capture the mathematical
content of statements of ordinary mathematics. In what follows, we shall take
this faithfulness for granted, at least for the specific mathematical theorems
discussed. Nevertheless, there are cases where we might doubt this faithfulness,
or at least worry about the commitments that attend certain types of coding, as
in the investigations of \citet{Kohlenbach2002} and \citet{Hunter2008}.

Despite its appeal, the standard view as it has been advanced in the literature
to date has a major weakness, namely that its central concept of a set existence
principle is left unanalysed, and thus the precise content of the view is
unclear. The goal of this article is therefore to provide such an analysis.
Our starting point is to note that accounts of the significance of reversals may
offer at least two kinds of explanation. The first is the significance of
reversals as a general matter: what do the equivalences proved in reverse
mathematics between theorems of ordinary mathematics and canonical subsystems of
second order arithmetic tell us? Let us call such explanations ``global''. The
second is the significance of reversals to a particular system, such as the
significance of reversals to $\wkl_0$ discussed above. Let us call explanations
of this sort ``local''. Depending on the account one offers, local explanations
may follow from global explanations, in part or in whole, or they may not. The
standard view is, prima facie, committed to a global explanation of the
significance of reversals, namely that they allow us to calibrate the strength
of set existence principles necessary to prove theorems of ordinary mathematics.
However, without a clear understanding of what set existence principles are, it
is unclear what the relationship between local and global explanations is under
the standard view. For example, it is commonly held that reversals to $\wkl_0$
demonstrate that proofs of the given theorem must involve some form of
compactness argument. In itself this is a local explanation, but without a
fuller account of what set existence principles are, it is left open whether it
follows from the global explanation of the significance of reversals in terms of
set existence principles, or whether it is genuinely local in character. As we
shall see, the answer to this question is highly sensitive to the particular
account of set existence principles one adopts.

\section{Set existence as comprehension}
\label{sec:sec_view}

One initially attractive way of cashing out the notion of a set existence
principle is taking it to be identical to the concept of a comprehension scheme.
I call this view ``set existence as comprehension'' or SEC.
A comprehension scheme is a schematic principle, each instance of which asserts
the existence of the (uniquely determined, by extensionality) set of all and
only those objects that satisfy a given predicate. In the case of second order
arithmetic, the objects are natural numbers, and the predicates concerned are
formulas of the language $\lt$ of second order arithmetic. Typically, a
comprehension scheme collectively asserts the existence of the sets defined by a
particular formula class, as follows.
Let $\Gamma$ be a set of formulas in the language $\lt$ of second order
arithmetic.
The \emph{$\Gamma$-comprehension scheme} consists of the universal closures of
all formulas of the form
\begin{equation}
    \exists{X} \forall{n} (n \in X \leftrightarrow \varphi(n))
    \label{eqn:gamma_ca}
    \tag{$\Gamma\text{-}\mathsf{CA}$}
\end{equation}
where $\varphi$ belongs to $\Gamma$ and $X$ is not free in $\varphi$.
If $\Gamma$ consists of all $\lt$-formulas, we obtain the full comprehension
scheme $\Pi^1_\infty\text{-}\mathsf{CA}$ of second order arithmetic, the
distinguishing axiom of full second order arithmetic $\z_2$. By restricting
$\Gamma$ to different formula classes, we obtain different comprehension
schemes. For example, if $\Gamma$ consists of formulas with only number
quantifiers, we obtain the arithmetical comprehension scheme, the distinguishing
axiom of the subsystem $\aca_0$. By restricting to formulas of the form
$\exists{k}\theta(k)$ where $\theta$ is quantifier-free ($\Sigma^0_1$ formulas),
we obtain the $\Sigma^0_1$-comprehension scheme; this is in fact equivalent to
the arithmetical comprehension scheme. By restricting to formulas of the form
$\forall{X}\theta(X)$ where $\theta$ is arithmetical, we obtain the $\Pi^1_1$-%
comprehension scheme, the distinguishing axiom of the subsystem $\pica_0$.
Some comprehension schemes do not follow this template to the letter, most
notably the recursive comprehension scheme (\ref{eqn:delta_0_1_ca}). This is a
subscheme of $\Sigma^0_1$-comprehension incorporating the further restriction
that each admissible formula be not merely $\Sigma^0_1$, but provably equivalent
to a $\Pi^0_1$ formula (i.e. of the form $\forall{k}\theta(k)$ where $\theta$ is
quantifier-free).

The idea that any given formal property (i.e. one defined by a formula of a
formal language properly applied to some domain) has an extension is a highly
credible basic principle, so long as appropriate precautions are taken to avoid
pathological instances. Second order arithmetic is a fragment of simple type
theory, which guards against Russell-style paradoxes, as formulas like
$X \not\in X$ are not well-formed and thus cannot appear in instances of
comprehension.
Moreover, comprehension schemes fall into straightforward hierarchies, with
increasingly strong comprehension schemes permitting the use of syntactically
broader classes of definitions. This harmonises with the fact that some theorems
are true even of the computable sets, while others entail the existence of
witnesses that are computability-theoretically highly complex, and strong axioms
are therefore needed in order to prove them. Such gradations can also be seen as
hierarchies of acceptability: if one denies that noncomputable sets exist then
recursive comprehension forms a natural stopping point; if one repudiates
impredicativity then arithmetical comprehension could be a good principle to
adopt; and so on.

Of the Big Five subsystems of second order arithmetic which are of primary
importance to reverse mathematics, only three are characterised by comprehension
schemes: $\rca_0$, $\aca_0$ and $\pica_0$. The intermediate systems $\wkl_0$ and
$\atr_0$ are obtained by adding further axioms (weak König's lemma in the former
case, arithmetical transfinite recursion in the latter) to a comprehension
scheme. Nevertheless, one might think that while weak König's lemma and
arithmetical transfinite recursion are not formulated as comprehension schemes,
they are nonetheless equivalent to comprehension schemes. The following result
of Dean and Walsh shows that this is not the case for weak König's lemma.%
\footnote{%
\label{fn:dw_proof}%
Dean and Walsh's proof of this theorem appeals to \citet*{SimTanYam2002}'s
result that $\wkl_0$ is conservative over $\rca_0$ for sentences of the form
$\forall{X}\exists{!Y}\varphi(X,Y)$, where $\varphi$ is arithmetical.
The theorem also follows from the non-existence of minimal $\omega$-models of
$\wkl_0$ and, as in the proof I give here, from basis theorems for $\Pi^0_1$
classes (\citet*{DiaDzhSoa2010} is a good reference for such basis theorems).
My thanks to the two anonymous referees for their suggestions towards a more
optimal proof.
}

\begin{thm}[Dean and Walsh, private communication]
    \label{thm:wkl_not_comp}
    No subset of the arithmetical comprehension scheme is equivalent over
    $\rca_0$ to weak König's lemma.
\end{thm}

\begin{proof}
    Assume for a contradiction that there is a set of arithmetical formulas
    $\Psi$ such that $\rca_0$ proves that $\Psi\text{-}\mathsf{CA}$ is
    equivalent to weak König's lemma.
    Without loss of generality, we may assume that there is a single
    arithmetical formula $\varphi(n, X)$ with only the displayed free variables
    such that
    \begin{equation*}
        \rca_0 \vdash
            \mathrm{WKL} \leftrightarrow \forall{X} C_\varphi(X),
    \end{equation*}
    where $C_\varphi(X)$ is the instance of arithmetical comprehension
    asserting that the set $\Set{ n \in \Nat | \varphi(n, X) }$ exists.

    Let $M$ be a countable $\omega$-model of $\wkl_0$, and thus of
    $\forall{X} C_\varphi(X)$. $\mathrm{WKL}$ is computably false, so there
    exist $X, Y \in M$ with $Y = \Set{ n \in \omega | \varphi(n, X) }$ and
    $X \Tlneq Y$.
    By cone avoidance for $\Pi^0_1$ classes, there exists a countable
    $\omega$-model $M'$ of $\wkl_0$ such that $X \in M'$ but $Y \not\in M'$.
    Since  $M'$ must also be a model of $C_\varphi(X)$, there exists $Z \in M'$
    such that $n \in Z \leftrightarrow \varphi(n, X)$ for all $n \in \omega$.
    But then by extensionality $Z = Y$, contradicting the fact that
    $Y \not\in M'$.
\end{proof}

\citet[pp.~29--30]{DeaWal2016} argue that weak König's lemma is therefore a
counterexample to SEC. Crucially, $\wkl_0$ is not merely a subsystem of second
order arithmetic that is not equivalent to a comprehension scheme: it is a
\emph{mathematically natural} one, since weak König's lemma is equivalent over
$\rca_0$ to the Heine--Borel covering lemma, Brouwer's fixed point theorem, the
separable Hahn--Banach theorem, and many other theorems of analysis and algebra.

\citet[p.~119]{Simpson2010} defines a subsystem of second order arithmetic as
being mathematically natural just in case it is equivalent over a weak base
theory to one or more ``core'' mathematical theorems. $\wkl_0$, $\aca_0$,
$\atr_0$ and $\pica_0$ are all mathematically natural systems, since each one is
equivalent over $\rca_0$ to many theorems from different areas of ordinary
mathematics.
This is a notion that seems to admit of degrees: some systems may, in virtue of
being equivalent to many core mathematical theorems, be more mathematically
natural than those which are only equivalent to a few such theorems. When a
claim of the form ``$S$ is a mathematically natural system'' is used in an
unqualified way in the rest of this article, it should be taken to mean that
$S$ meets the minimum requirement of being equivalent to at least one core
mathematical theorem.
The notion of mathematical naturalness appears to give us a partial answer to
the question of the significance of reversals: by proving an equivalence between
a theorem of ordinary mathematics $\tau$ and a subsystem of second order
arithmetic $S$, we thereby demonstrate that $S$ is a mathematically natural
system. However, this still leaves us in the dark about the significance of the
reversal for the theorem $\tau$: what important property of this theorem of
ordinary mathematics do we come to know when we prove its equivalence over a
weak base theory to $S$, that we did not know before? It is this question that
the standard view, and thus the SEC account as an explication of the standard
view's central theoretical notion, attempts to answer.

Dean and Walsh's argument that SEC fails runs as follows: since weak König's
lemma is neither a comprehension scheme, nor equivalent to one, it cannot be a
set existence principle (as by SEC, set existence principles are just
comprehension schemes). So the significance of a reversal from a theorem $\tau$
to weak König's lemma cannot lie in the comprehension scheme that is both
necessary and sufficient to prove it, since there is no such scheme. Either the
significance of reverse mathematics does not lie in the set existence principles
which theorems reverse to, or the set existence as comprehension view is false.
Proponents of the standard view are thus on shaky ground. They must adopt a more
sophisticated way of spelling out their core contention, or abandon the idea
that the significance of reversals lies in set  existence principles. We can
also understand Dean and Walsh's point in terms of the distinction between local
and global explanations: SEC fails to offer a local explanation of the
significance of reversals to weak König's lemma, since it is not a
comprehension scheme, despite the fact that reversals to $\wkl_0$ are clearly
significant. The global explanation of the significance of reversals offered by
the SEC account is therefore inadequate.

One response to Dean and Walsh's argument is to endorse a more expansive
conception of set existence principles, based on the fact that Weak König's
lemma and arithmetical transfinite recursion are equivalent to schematic principles of a different sort, namely \emph{separation schemes}.
The separation scheme for a class of formulas $\Gamma$ holds that if two
formulas $\varphi, \psi \in \Gamma$ are provably disjoint, then there exists a
set whose members include every $n \in \Nat$ such that $\varphi(n)$, and exclude
every $n \in \Nat$ such that $\psi(n)$. Weak König's lemma is equivalent over
$\rca_0$ to $\Sigma^0_1$-separation, while arithmetical transfinite recursion is
equivalent over $\rca_0$ to $\Sigma^1_1$-separation.

\begin{dfn}[separation scheme]
    Let $\Gamma$ be a class of formulas, possibly containing free variables.
    The $\Gamma$-separation scheme, $\Gamma\text{-}\mathsf{SEP}$,
    consists of the universal closures of all formulas of the form
    \begin{equation*}
        \forall{n}(\neg(\varphi(n) \wedge \psi(n)))
        \rightarrow
        \exists{X}\forall{n} (
            (\varphi(n) \rightarrow n \in X)
            \wedge
            (\psi(n) \rightarrow n \not\in X)
        ),
        \tag{$\Gamma\text{-}\mathsf{SEP}$}
    \end{equation*}
    where $\varphi, \psi \in \Gamma$ with $X$ not free.
\end{dfn}

One response to Dean and Walsh's argument is to endorse the following more
expansive conception of set existence principles: both comprehension schemes and
separation schemes are set existence principles.

In line with our existing terminology, we call the proposal that both
comprehension schemes and separation schemes are set existence principles SECS.
This solves the immediate problem with the SEC view, since each of the Big Five
are equivalent over $\rca_0$ to either a comprehension scheme or a separation
scheme. But although SECS accommodates both weak König's lemma and arithmetical
transfinite recursion, and thus evades the counterexamples that sink SEC, it
does so at the price of a seemingly ad hoc modification to the view.

\section{Conceptual constraints}
\label{sec:constraints}

The arguments levelled against the SEC account and its variants tacitly appeal
to different constraints which the concept of a set existence principle should
satisfy, if it is to play a role in explaining the significance of reversals. I
shall now attempt to make these constraints explicit, by presenting three
conditions which any satisfactory account of the concept of a set existence
principle should meet, together with some reasons to believe that these
conditions are plausible. I shall then show how the SECS account meets two of
the stated conditions, but fails to satisfy a third.
In the following we take $\mathcal{A}$ to be some account of the concept of a
set existence principle (e.g. SEC, SECS) with an associated extension
$\mathcal{S(A)}$, consisting of the set of subsystems of second order arithmetic
which, according to the account $\mathcal{A}$, express set existence principles.

\begin{enumerate}
    \item[(1)] \textsc{Nontriviality.}
        \label{se_nontriviality}
        \emph{Not all subsystems of second order arithmetic are classified as
        set existence principles by $\mathcal{A}$, i.e. at least one subsystem
        is not a member of $\mathcal{S(A)}$.}
    \item[(2)] \textsc{Comprehensiveness.}
        \label{se_comprehensiveness}
        \emph{If $S$ is a mathematically natural subsystem of second order
        arithmetic, and $S$ actually expresses a set existence principle,
        then $\mathcal{A}$ classifies $S$ as expressing a set existence
        principle: $S \in \mathcal{S(A)}$.}
    \item[(3)] \textsc{Unity.}
        \label{se_unity}
        \emph{The subsystems of second order arithmetic which $\mathcal{A}$
        classifies as expressing set existence principles, i.e. the members of
        $\mathcal{S(A)}$, are conceptually unified.}
\end{enumerate}

Consider some account of set existence principles $\mathcal{A}$. Such an account
should lend substance to the claim that the significance of reversals lies in
the set existence principles necessary to prove theorems of ordinary
mathematics, by providing some degree of analysis of the concept of a set
existence principle, and a way to determine whether a subsystem of second order
arithmetic expresses a set existence principle (for example, the SEC account
does this by stating that set existence principles are just comprehension
schemes).
If $\mathcal{A}$ does not satisfy the nontriviality condition
(\ref{se_nontriviality}) then it cannot do this. There are many statements of
second order arithmetic, such as purely arithmetical ones containing no set
quantifiers, that prima facie are not set existence principles. Violating the
nontriviality condition therefore entails failing to provide a theory that is
truly an account of set existence principles at all.

If $\mathcal{A}$ does not meet the comprehensiveness condition
(\ref{se_comprehensiveness}), then it not only fails to include prima facie set
existence principles, but ones that are equivalent to core mathematical
theorems. $\mathcal{A}$ would therefore fail to provide a satisfactory
reconstruction of the concept of a set existence principle in reverse
mathematics, and the reconstruction offered would be unable to play its intended
role in the standard view.
The antecedent of the condition is relatively strong, since it requires that a
system both express a set existence principle and be mathematically natural. As
noted in §\ref{sec:introduction}, some core mathematical theorems reverse to
principles that are prima facie not set existence principles, such as Hilbert's
basis theorem which is equivalent over $\rca_0$ to $\mathrm{WO}(\omega^\omega)$.
It is thus not to $\mathcal{A}$'s detriment---and, indeed, may be to its credit%
---if it does not classify such systems as set existence principles.
On the other hand, there are also systems that prima facie express set existence
principles, but that are not mathematically natural, or at least are not known
to be. It is reasonable, but not required, for $\mathcal{A}$ to include such
systems. The $\Delta^1_1$-comprehension scheme is a set existence principle
according to the SEC account, but it has not been proved equivalent to any
theorem of ordinary mathematics. Various forms of the axiom of choice, such as
the systems $\Sigma^1_1\text{-}\mathsf{AC}$ and $\Sigma^1_1\text{-}\mathsf{DC}$,
also appear to be set existence principles that are not mathematically natural,
but in these cases the SEC account does not classify them as set existence
principles.%
\footnote{%
See definition VII.6.1 of \citet[p.~294]{Simpson2009} for details of these
systems.
}
The comprehensiveness condition (\ref{se_comprehensiveness}) permits such cases
because in order for an account of the concept of a set existence principle to
elucidate the standard view of reverse mathematics, it does not need to say
anything about the status of systems which are not equivalent to any theorem of
ordinary mathematics.
With these points in mind, we can see clearly why the fact that SEC does not
include weak König's lemma is so problematic. Resisting its mathematical
naturalness is difficult, since it is equivalent to dozens of core mathematical
theorems in analysis, algebra, logic, and combinatorics. Moreover, it is prima
facie a set existence principle, albeit in a somewhat different sense to
comprehension schemes. The standard view that reversals demonstrate the set
existence axioms necessary to prove theorems of ordinary mathematics would not be
the standard view in the field if it did not include $\wkl_0$.

Any failure of $\mathcal{A}$ to meet the unity condition (\ref{se_unity}) has a
somewhat different character. The standard view is an attempt to provide a general
account of the significance of reversals, one that does not make overt reference
to particular systems. Such generality requires that the systems which
$\mathcal{A}$ countenances as set existence principles have some features in
common. For example, while recursive comprehension, arithmetical comprehension
and $\Pi^1_1$-comprehension are all different, the SEC account still satisfies
the unity condition precisely because they are all comprehension schemes, and it
can offer a theory under which all comprehension schemes can legitimately be
considered to be set existence principles. If $\mathcal{A}$ does not satisfy the
unity condition then it cannot be considered as offering a satisfactory theory
of set existence principles, even if it is extensionally adequate in the sense
of satisfying conditions (\ref{se_nontriviality}) and
(\ref{se_comprehensiveness}). If no account meeting this condition can be found
then we are reduced to merely offering local explanations of the significance of
reversals to particular systems, rather than a global explanation of the
significance of reverse mathematical results. Moreover, an inability to provide
a unified account of the concept of a set existence principle would cast doubt
on whether there is indeed a unified concept that can play the explanatory role
the standard view requires of it.

Unity, however, comes in degrees, and so accounts of the concept of a set
existence principle can satisfy it in a stronger or weaker way. When there are
strong connections between the different systems considered to be set existence
principles, then the account under which those systems are considered to be set
existence principles can be said to strongly satisfy the unity condition. In
such cases it seems reasonable to expect that the significance of reversals to a
particular system $S$ will, in large part, be given in terms of the account of
set existence principles, rather than in terms of specific properties of $S$
that are at substantial variance to other set existence principles.
The SEC account exhibits this property: different comprehension schemes are
learly the same type of principle, and can be obtained by simple syntactic
restrictions on a stronger principle, namely the full comprehension scheme.
evertheless, requiring that any theory of set existence principles satisfies
the unity condition as strongly as the SEC account is an onerous requirement
that may well be impossible to meet in a theory that also satisfies the
comprehensiveness condition. Allowing for theories to satisfy the unity
requirement only weakly, and have different set existence principles bear a mere
family resemblance to one another, rather than be strictly of the same type of
axiom in some strong syntactic sense, seems a reasonable relaxation of the
ondition.

Since the comprehensiveness condition (\ref{se_comprehensiveness}) is a
conditional and not a biconditional, it does not require that every subsystem
that expresses a set existence principle according to $\mathcal{A}$ also
expresses set existence principle by the lights of our naïve, informal, or
intuitive understanding of the concept. This reflects the kind of counterexample
encountered in \S\ref{sec:sec_view} and \S\ref{sec:secs_view}, namely
$\mathrm{WKL}$ and $\mathrm{WWKL}$. These are sentences that prima facie express
set existence principles, but are not classified as such, by the SEC and SECS
accounts respectively.
Contrastingly, we have few examples of sentences that are classified as set
existence principles by the accounts currently on the table, and yet prima facie
do not express set existence principles.%
\footnote{%
Induction axioms are one possible exception; see footnote \ref{fn:induction}
for details.
}
Moreover, such cases do not seem to undermine the coherence or usefulness of the
concept of a set existence principle in the way that counterexamples like
$\mathrm{WKL}$ and $\mathrm{WWKL}$ do, because they do not give us the same
sense that our theory of set existence principles is somehow inadequate to the
data. The nontriviality condition (\ref{se_nontriviality}) and the unity
condition (\ref{se_unity}) are therefore formulated so as to constrain how
extensive the class of subsystems classified as set existence principles by
$\mathcal{A}$ is, but the conditions as a whole do not require that
$\mathcal{A}$ provide a list of subsystems that is coextensive with our informal
understanding of the concept of a set existence principle.

\section{Set existence as comprehension and separation}
\label{sec:secs_view}

Admitting separation schemes as set existence principles is, prima facie, an ad
hoc modification of the SEC view that weakens one of the main strengths of the
SEC view, namely its strong satisfaction of the unity condition
(\ref{se_unity}). The primary point of difference between separation and
comprehension schemes is that as straightforward definability axioms,
comprehension schemes tell us which particular sets exist. Separation schemes,
on the other hand, do not always do so: an axiom asserting the mere existence of
a separating set may not pin down a particular set as the witness for this
assertion. A striking theorem in this vein is that the only sets which every
$\omega$-model of $\Sigma^0_1$-separation (i.e. $\wkl_0$) has in common are the
computable ones. This illustrates \citet[p.~235]{Friedman1975}'s point that
``Much more is needed to define explicitly a hard-to-define set of integers than
merely to prove their existence.''

To rebut the charge that SECS is ad hoc, and show that it does after all satisfy
the unity requirement, we must show that there is some degree of conceptual
commonality between comprehension schemes and separation schemes. Let us first
note that one important comprehension scheme is a separation scheme:
$\Delta^0_1$-comprehension is equivalent, over a weak base theory, to the
$\Pi^0_1$-separation scheme.
Moreover, following \citet{Lee2014}, we can treat the entire Big Five in a
unified way by understanding them as \emph{interpolation schemes}.

\begin{dfn}[interpolation scheme]
    Let $\Gamma$ and $\Delta$ be sets of $\lt$-formulas, possibly with free
    variables.
    The \emph{$\Gamma\text{-}\Delta$ interpolation scheme},
    $\Gamma\text{-}\Delta\text{-}\mathsf{INT}$,
    consists of the universal closures of all formulas of the form
    \begin{equation*}
        \begin{aligned}
            & \forall{m}(\varphi(m) \rightarrow \psi(m))
            \\ \rightarrow \;
            & \exists{X}\forall{m} (
                (\varphi(m) \rightarrow m \in X)
                \wedge
                (m \in X \rightarrow \psi(m))
            )
        \end{aligned}
        \tag{$\Gamma\text{-}\Delta\text{-}\mathsf{INT}$}
    \end{equation*}
    where $\varphi \in \Gamma$ and $\psi \in \Delta$, and $X$ is not free in
    either formula.
\end{dfn}

All of the Big Five are equivalent to interpolation schemes:
$\rca_0$ is equivalent to $\Pi^0_1\text{-}\Sigma^0_1\text{-}\mathsf{INT}$;
$\wkl_0$ to $\Sigma^0_1\text{-}\Pi^0_1\text{-}\mathsf{INT}$;
$\aca_0$ to $\Sigma^0_1\text{-}\Sigma^0_1\text{-}\mathsf{INT}$;
$\atr_0$ to $\Sigma^1_1\text{-}\Pi^1_1\text{-}\mathsf{INT}$;
and $\pica_0$ to $\Sigma^1_1\text{-}\Sigma^1_1\text{-}\mathsf{INT}$. This should
go at least some way towards ameliorating our worry that SECS fails to satisfy
the unity condition (\ref{se_unity}), since we can now see that both
comprehension schemes and separation schemes are actually interpolation schemes.

Mere syntactic unity should not by itself convince us of the conceptual unity of
comprehension schemes and separation schemes; after all, a sufficiently broad
syntactically specified class of sentences will eventually contain all
statements. The notion of an interpolation scheme is, however, relatively narrow
and it is not hard to see that it is a straightforward generalisation of the
concepts of separation and comprehension. The comprehension scheme for some
class of formulas $\Phi$ can be derived from the $\Phi\text{-}\Phi$
interpolation scheme, since for any instance of comprehension we can use the
given formula in both places in the interpolation scheme and thus derive the
comprehension instance. Separation schemes, on the other hand, arise when given
some formula class $\Delta$, the formula class $\Gamma$ consists of the
negations of the formulas in $\Delta$, such as when $\Delta = \Sigma^0_1$ and
$\Gamma = \Pi^0_1$.%
\footnote{%
As noted in §\ref{sec:sec_view}, $\rca_0$ deviates somewhat from the standard
template for comprehension schemes, and for this reason is not equivalent to an
interpolation scheme of the type $\Phi\text{-}\Phi\text{-}\mathsf{INT}$ where
both sets of formulas are the same.
By analogy with the case of $\rca_0$, one might expect the $\Delta^1_1$-%
comprehension scheme to be equivalent to
$\Pi^1_1\text{-}\Sigma^1_1\text{-}\mathsf{INT}$,
i.e. the $\Pi^1_1$-separation scheme. However, this is not the case, as
\citet{Montalban2008} proved. Since $\dica_0$ is not known to be a
mathematically natural subsystem of $\z_2$, this fact does not by itself seem to
pose a particular problem for the SECS view.
}

However, even if we grant that SECS satisfies the unity condition
(\ref{se_unity}), it still fails to offer a satisfactory theory of set existence
principles, since there is a mathematically natural counterexample which shows
that it does not satisfy the comprehensiveness condition
(\ref{se_comprehensiveness}), namely the axiom known as \emph{weak weak König's
lemma}.
This principle was introduced in \citet{Yu1987}, and as the name suggests, it is
a further weakening of weak König's lemma, obtained by restricting weak König's
lemma to trees with positive measure.

\begin{dfn}[weak weak König's lemma]
    \emph{Weak weak König's lemma} ($\mathrm{WWKL}$) is the statement that
    if $T$ is a subtree of $2^{<\Nat}$ with no infinite path, then
    \begin{equation*}
        \lim_{n \rightarrow \infty}
            \frac{\left|\Set{
                \sigma \in T | \mathrm{lh}(\sigma) = n
            }\right|}{2^n} = 0.
    \end{equation*}
    The system $\wwkl_0$ is given by adjoining $\mathrm{WWKL}$ to the axioms of
    $\rca_0$.
\end{dfn}

$\wwkl_0$ is strictly intermediate between $\rca_0$ and $\wkl_0$
\citep{YuSim1990}, and is equivalent over $\rca_0$ to a number of theorems from
measure theory, such as a formal version of the Vitali Covering Theorem
\citep*{BroGiuSim2002}; the countable additivity of the Lebesgue measure
\citep{YuSim1990}; and the monotone convergence theorem for the Lebesgue
measure on the closed unit interval. A survey of results in this area is given
in \citet[§X.1]{Simpson2009}. These equivalences show that $\mathrm{WWKL}$ is
mathematically natural in Simpson's sense. By the comprehensiveness condition
(\ref{se_comprehensiveness}) we should therefore expect a good account of set
existence principles to include it.

$\wwkl_0$ is not equivalent over $\rca_0$ to any subset of the arithmetical
comprehension scheme.%
\footnote{%
This follows, for example, from the conservativity theorem of
\citet*{SimTanYam2002} mentioned in footnote \ref{fn:dw_proof}.
}
\citet[§2, pp.~172--3]{YuSim1990} proved that not every model of $\wwkl_0$ is a
model of $\wkl_0$. Their proof involves the construction of what is known as a
\emph{random real model}, and the properties of this model show that $\wwkl_0$
is not equivalent over $\rca_0$ to any subscheme of $\Sigma^0_1$-separation.
Since $\wwkl_0$ is not a separation scheme, it is a mathematically natural
subsystem of $\z_2$ that SECS cannot accommodate. SECS therefore fails to
satisfy the comprehensiveness condition (\ref{se_comprehensiveness}).

\begin{thm}[\citealt{YuSim1990}]
    Weak weak König's lemma is not equivalent over $\rca_0$ to any subscheme of
    the $\Sigma^0_1$-separation scheme.
\end{thm}

A virtue that it would be reasonable to expect of any account of set existence
principles is the ability to incorporate the discovery of new subsystems of
second order arithmetic which turn out to be equivalent to theorems of ordinary
mathematics. Banking on SEC or its extensions amounts to a bet that all such new
systems will be comprehension schemes or separation schemes. The discovery of
$\mathrm{WWKL}$ and the role it plays in the reverse mathematics of measure
theory shows that such optimism is unfounded even for the mathematically natural
systems which are already known. In the next section I will advance an account
of set existence principles which does not suffer from this weakness.

\section{Closure conditions}
\label{sec:closure_view}

In a sense the term ``set existence principles'' is an unfortunate one, since
it implies that the relevant principles simply assert the existence of some
sets, independently of the other axioms of the theory. A better term, which more
accurately captures the nature of these axioms, is ``closure conditions''.
These are not bare or unconditional statements of set existence, but conditional
principles. In general they hold that given the existence of a set $X \in \PowN$
with certain properties, there exists another set $Y \in \PowN$ standing in some
relation to $X$.
Weak König's lemma, for example, asserts that $\PowN$ is closed under the
existence of paths through infinite subtrees of $2^{<\Nat}$. This is a
conditional set existence principle since it requires that there be such trees
in the first place. In the absence of other suitable existence axioms, weak
König's lemma alone would not allow us to prove the existence of any sets at
all, and is thus a set existence principle only relative to other axioms such
as recursive comprehension.

On the other hand, comprehension schemes appear at first blush to be set
existence principles tout court. Nevertheless, they too are better understood as
closure conditions, because the comprehension schemes used in reverse
mathematics all admit parameters. Comparing the standard formulation of
recursive comprehension (in which parameters are allowed) with the parameter-%
free version makes this clear.
The parameter-free recursive comprehension scheme asserts the existence of those
sets definable in a $\Delta^0_1$ way, without reference to any other sets. But
recursive comprehension with parameters instead asserts that $\PowN$ is closed
under relative recursiveness: if $X \subseteq \Nat$ exists, so does every
$Y \Tleq X$.
While comprehension schemes do have a different flavour to other closure
conditions, they can often be characterised in equivalent ways which more
closely hew to the model described above for weak König's lemma.
Arithmetical comprehension, for example, is equivalent over $\rca_0$ to König's
lemma, the statement that every finitely branching infinite subtree of
$\Nat^{<\Nat}$ has an infinite path through it.
$\pica_0$ is equivalent over $\rca_0$ to the statement that for every tree
$T \subseteq \Nat^{<\Nat}$, if $T$ has a path then it has a leftmost path.

With these points in mind, I propose that the best way to understand the concept
of a set existence principle in reverse mathematics is by means of the concept
of a closure condition on the powerset of the natural numbers. Reformulating the
standard view of reverse mathematics in these terms, we come to the view that
the significance of a provable equivalence between a theorem of ordinary
mathematics $\tau$ and a subsystem $S$ of second order arithmetic lies in
telling us what closure conditions $\PowN$ must satisfy in order for $\tau$
to be true. This is a bit of a mouthful, so we shall adopt the following slogan
as an abbreviation for the view: \emph{reversals track closure conditions}.

Views of this sort have been stated elsewhere, most notably by
\citet[p.~451]{Feferman1992} who identifies set existence principles with
closure conditions in his discussion of the mathematical existence principles
justified by empirical science.
Indeed, already in \citet[p.~8]{Feferman1964} we find that work in the spirit of
Weyl's predicative analysis ``isolates various closure conditions on a
collection of real numbers which are necessary to obtain [results such as the
Bolzano--Weierstraß and Heine--Borel theorems], in this case closure under
arithmetical definability.''
Similar positions have also been taken in the reverse mathematics literature,
for example by \citet*[p.~2]{DorDzhHir2015}, who write that each subsystem
studied in reverse mathematics ``corresponds to a natural closure point under
logical, and more specifically, computability-theoretic, operations'', and by
\citet*[p.~864]{ChoSlaYan2014}, who write that ``Ultimately, we are attempting
to understand the relationships between closure properties of $2^{\Nat}$''.
\citet{Shore2010} also states that the Big Five correspond to recursion-%
theoretic closure conditions: this is clearly something in the air. However,
none of these authors make precise what they mean by a closure condition, nor
draw out the consequences of this view, although \citet*{ChoSlaYan2014} consider
it to have consequences for the practice of reverse mathematics, for instance
in showing that $\omega$-models have a particular importance.

In order to clarify the content of the view that reversals track closure
conditions, let us distinguish two things: a closure condition in itself, and
the different axiomatizations of that closure condition. Closure conditions are
extensional: they are conditions on the powerset which assert that it is closed
in some way, for example under arithmetical definability. Axiomatisations of
closure conditions are intensional: one and the same closure condition will
admit of infinitely many different axiomatizations. For example, the Turing jump
operator gives rise to a closure condition, of which some of the better-known
axiomatizations are (modulo the base theory $\rca_0$): the arithmetical
comprehension scheme; König's lemma; and the Bolzano--Weierstraß theorem.

The upshot of this distinction is that by proving reversals we show that
different theorems of ordinary mathematics correspond to the same closure
conditions. The significance of reversals thus lies, at least to a substantial
extent, in placing these theorems in a hierarchy of well-understood closure
conditions of known strength. Note also that there is a duality here: an
equivalence between a theorem $\tau$ and a system $S$ tells us something about
$\tau$, namely its truth conditions in terms of what closure condition must hold
for it to be true, but it also tells us something about the closure condition
itself, namely how much of ordinary mathematics is true in $\PowN$ when that
closure condition holds.

This account incorporates a substantial relativism to the base theory, in at
least two distinct ways. The first is that, as noted above, some closure
conditions such as weak König's lemma are set existence principles only in a
conditional sense, and using them to prove the existence of sets relies on them
being used in conjunction with other axioms. The second is that the equivalences
which I take to show that different statements express the same closure
condition must be proved in some base theory. Typically, this will be the usual
base theory for reverse mathematics, $\rca_0$, but in some cases a stronger base
theory is required in order to prove that two statements are equivalent and thus
express the same closure conditions.

The view that reversals track closure conditions has some marked advantages over
the SEC account and its variants such as SECS. Most notably, it can accommodate
all of the counterexamples discussed so far. Weak König's lemma is a closure
condition, and thus we improve on the SEC account; but so is weak weak König's
lemma, and thus the new account also succeeds where the SECS account fails.
Other principles which have been studied in reverse mathematics---arithmetical
transfinite recursion, choice schemes, and many others---can all be understood
as expressing closure conditions on $\PowN$. Moreover, this account will also
accommodate any similar principle discovered to be equivalent to a theorem of
ordinary mathematics.
In order to determine whether or not the view goes beyond this apparent
improvement over the SEC and SECS accounts, and provides a satisfactory account
of the significance of reversals in general, we return to the three conditions
that I argued in §\ref{sec:constraints} any account of set existence principles
should satisfy:
nontriviality (\ref{se_nontriviality}),
comprehensiveness (\ref{se_comprehensiveness}),
and unity (\ref{se_unity}).

By analysing the notion of a set existence principle in terms of closure
conditions on the powerset of the natural numbers, the account clearly offers a
unified picture of what set existence principles are. However, the notion of a
closure condition is a very general one. At first glance, the view seems
committed to the idea that every $\Pi^1_2$ sentence expresses a closure
condition. By way of contrast, the specificity of the concept of a comprehension
scheme means that the SEC account strongly satisfies the unity condition. But
this very feature undermines its suitability as an analysis of the concept of a
set existence principle, since it fails to be sufficiently comprehensive, as the
existence of striking counterexamples such as weak König's lemma illustrates.
We therefore must conclude that although the view that reversals track closure
conditions satisfies the unity condition, it only satisfies it in a relatively
weak sense. As such, it is reasonable to wonder to what degree the view can
offer a compelling explanation of the significance of reversals, since if it is
easy for an $\lt$-sentence to be considered as expressing a closure condition,
then it is unclear what light this can shed on the importance of particular
reversals.
There is, however, a clear sense in which all closure conditions are the same
kind of thing: if weak König's lemma and Ramsey's theorem for pairs are not in
exactly the same class of principles, they certainly bear a family resemblance
to one another. With this in mind I contend that weakly satisfying the unity
condition is sufficient to make the view that reversals track closure conditions
a viable account, and moreover that weakly satisfying this condition is all one
can expect of an account of set existence principles that accommodates not only
$\mathrm{WKL}$ and $\mathrm{WWKL}$, but also the constellation of combinatorial
and model-theoretic statements in the Reverse Mathematics Zoo.

The generality of the view also undermines, to some extent, its ability to
satisfy the nontriviality condition (\ref{se_nontriviality}).
Arithmetical statements cannot be considered as expressing closure conditions,
and thus according to the account they do not express set existence principles.
Neither do $\Pi^1_1$ statements, such as those expressing that a given recursive
ordinal $\alpha$ is wellordered. Nevertheless, it is hard to escape from the
conclusion that at least every $\Pi^1_2$ statement should be considered a
closure condition. After all, it is the very form of these statements---which
assert that for every set $X \subseteq \Nat$ of a certain sort, there exists a
set $Y \subseteq \Nat$ of a different sort---that brought us to consider the
view that reversals track closure conditions in the first place.

This point holds more generally: thus far we have only considered closure
conditions with $\Pi^1_2$ formulations, but not even all of the Big Five have
$\Pi^1_2$ definitions. In particular, $\pica_0$ is not equivalent over $\atr_0$
to any $\Pi^1_2$ statement, although it is straightforwardly expressed as a
$\Pi^1_3$ sentence. There are even theorems that exceed the strength of
$\Pi^1_1$-comprehension, such as ``Every countably based MF space which is
regular is homeomorphic to a complete separable metric space'', which is
equivalent to $\Pi^1_2$-comprehension \citep{MumSim2005}. Such theorems are not
expressible as $\Pi^1_3$ statements, so we must consider yet more complex
sentences as also expressing closure conditions if we are to bring them into the
account.

When we restrict our attention to the closure conditions expressed by $\Pi^1_2$
sentences, the conditions involved are arithmetical, and thus indicate that
$\PowN$ is closed under the existence of sets which are definable without
quantification over the totality of sets of natural numbers. In other words,
$\Pi^1_2$ closure conditions are predicative. On an intuitive level, we can
understand these predicative closure conditions as local conditions, in the
following sense.
Suppose we have a class of sets $\mathcal{C} \subseteq \PowN$ and a closure
condition which states that for every $X$ such that $\psi(X)$, there exists a
$Y$ such that $\varphi(X,Y)$, where $\psi$ and $\varphi$ are arithmetical.
If there is such an $X \in \mathcal{C}$ such that there is no corresponding
$Y \in \mathcal{C}$, then $\mathcal{C}$ does not satisfy the instance of the
closure condition for $X$, and must be expanded to a class $\mathcal{C}'
\supseteq \mathcal{C}$ such that $Y \in \mathcal{C}'$ for some $Y$ such that
$\varphi(X, Y)$. Satisfying the instance of the closure condition for $X$ thus
depends only on the existence of $Y$, and the arithmetical properties of $X$ and
$Y$.
For impredicative set existence principles such as $\Pi^1_1$ or $\Pi^1_2$
comprehension, the situation is more complex, since if a class of sets
$\mathcal{C}$ fails to satisfy such a principle, then the new sets that must be
added to $\mathcal{C}$ in order to satisfy it depend not just on the single set
$X$, but also on the entirety of elements of the expanded class $\mathcal{C}'
\supseteq \mathcal{C}$. Impredicative closure conditions are thus
global conditions rather than local ones.
The more general, impredicative notion of closure at work here is thus distinct
from, and stronger than, that which applies in the case where the conditions
are arithmetical. Nevertheless, it is still recognisably a notion of closure:
an impredicative closure condition requires that for every set $X$ there exists
a $Y$ bearing some relation to $X$, although that relation may be given in terms
of not just arithmetical properties of $X$ and $Y$, but properties of all sets
of natural numbers.

Taking these considerations to their natural conclusion, it seems that
we should grant that all $\Pi^1_n$ sentences, for $n \geq 2$, express closure
conditions (for brevity's sake we shall sometimes refer to such sentences as
$\Pi^1_{n \geq 2}$ sentences in the remainder of this article).
Sentences of lower complexity (for example, arithmetical sentences) can be
transformed into $\Pi^1_2$ sentences by the addition of dummy quantifiers. To
work around this difficulty, we impose the following constraint on the
$\Pi^1_{n \geq 2}$ sentences which we take to express closure conditions.
For all $n \geq 1$, we say that a sentence $\varphi$ is \emph{essentially
$\Pi^1_n$} (respectively, \emph{essentially $\Sigma^1_n$})
if, over an appropriate base theory $B$, it is equivalent to a sentence that is
$\Pi^1_n$ (respectively, $\Sigma^1_n$), and not equivalent to any $\Sigma^1_n$
sentence (respectively, $\Pi^1_n$), or to any sentence of lower complexity.
The base theory $B$ should not prove $\varphi$, but otherwise it should be as
strong as possible.%
\footnote{%
The term ``essentially $\Pi^1_n$ ($\Sigma^1_n$) formula'' is used in the
literature on subsystems of second order arithmetic to denote a different,
formally specified formula class
(e.g. \citet[definition VIII.6.1, p.~348]{Simpson2009}).
In this paper, however, the term ``essentially $\Pi^1_n$ ($\Sigma^1_n$)
sentence'' will only be used in the sense of the definition just given.
}
These requirements are intended to ensure that if a sentence $\varphi$ does not
express a closure condition, then a sufficiently strong base theory is available
to prove this fact.
For example, consider the following sentence $\psi$:
``$\mathrm{WKL} \wedge \text{$0'$ exists}$''. $\psi$ is $\Pi^1_2$, but not
essentially $\Pi^1_2$. While $\rca_0$ does not prove $\psi$, neither does
$\wkl_0$, since weak König's lemma does not imply the existence of the Turing
jump. However, since $\wkl_0$ proves the first conjunct, it proves that $\psi$
is equivalent to the $\Sigma^1_1$ sentence ``$0'$ exists''. $\psi$ is thus
not essentially $\Pi^1_2$, and does not express a closure condition. This seems
like the correct outcome, since $\psi$ only asserts that $0'$ exists, rather
than making the more general (and essentially $\Pi^1_2$) claim that for all $X$,
$X'$ exists.
With the notion of essentially $\Pi^1_n$ and $\Sigma^1_n$ sentences in hand,
we can give a more precise account of which sentences express closure
conditions, namely all and only the essentially $\Pi^1_{n \geq 2}$ sentences.

\citet[pp.~29--30]{DeaWal2016} express the worry that there might be no broader
notion of set existence principle that includes both comprehension schemes and
principles like weak König's lemma and arithmetical transfinite recursion,
beyond that of accepting all $\Pi^1_2$ sentences as expressing set existence
principles. As argued in §\ref{sec:secs_view}, there is a broader notion (the
SECS view) that includes weak König's lemma and arithmetical transfinite
recursion, but this broader notion remains vulnerable to mathematically natural
counterexamples like $\mathrm{WWKL}$. In endorsing the view that all and only
the essentially $\Pi^1_{n \geq 2}$ sentences express closure conditions, I am
biting this bullet and broadly concurring with Dean and Walsh that there is no
viable broader notion of set existence principle beyond that encompassing all
essentially $\Pi^1_2$ sentences (and worse, all essentially $\Pi^1_n$
sentences as well, for $n \geq 3$). Since the essentialness restriction rules
out all sentences in the class $\Sigma^1_1 \cup \Pi^1_1$, as well as many
others, the situation is not disastrous: measuring set existence is not ``just
the same as sorting out the very fine-grained equivalence classes of mutual
derivability'' \citep[p.~30]{DeaWal2016}, although it is not so far away from
this either.

The restriction to essentially $\Pi^1_{n\geq2}$ sentences also rules out
wellordering statements such as $\mathrm{WO}(\omega^\omega)$, which seems like
the correct outcome, but other cases are less clear-cut.
According to \citet[remark II.3.11, p.~71--2]{Simpson2009},
$\Sigma^0_n$-induction is a set existence principle, because it is equivalent to
bounded $\Sigma^0_n$-comprehension (i.e. comprehension for finite
$\Sigma^0_n$-definable sets). However, since $\Sigma^0_n$-induction can (for any
$n \in \omega$) be axiomatized by a $\Pi^1_1$ sentence, it does not express a
closure condition, and hence according to the present account it is not a set
existence principle. Moreover, we can also run Simpson's argument in reverse:
since bounded $\Sigma^0_n$-comprehension is equivalent to $\Sigma^0_n$-%
induction, it does not (contrary to appearances) express a set existence
principle after all. As induction axioms have a very different character to the
closure conditions we have been considering in the rest of this paper, it seems
reasonable to consider them as distinct kinds of axiom. Simpson appears to
classify induction axioms as set existence principles primarily to further
bolster his claim that reverse mathematics shows us what set existence
principles are necessary to proving theorems of ordinary mathematics. In light
of the necessity of axioms such as $\mathrm{WO}(\omega^\omega)$ to proving
results like Hilbert's basis theorem, it does not seem overly problematic to
take the line that reverse mathematics shows that not only set existence
principles, but also other kinds of axiom, are required in order to prove
theorems of ordinary mathematics. Amongst these we can include both wellordering
statements and induction axioms.%
\footnote{%
\label{fn:induction}
One lacuna is induction schemes where the formula class for which induction is
permitted is not arithmetical, such as $\Sigma^1_1$-induction. These appear to
be equivalent to essentially $\Pi^1_{n\geq2}$ sentences, and thus express
closure conditions. This is an unfortunate asymmetry with the arithmetical
induction scheme and its subschemes such as $\Sigma^0_2$-induction. This
asymmetry can be overcome by requiring that all base theories include full
induction, or by only considering $\omega$-models as \citet{Shore2010} does,
but these approaches incur other costs.
}

A different kind of difficulty is posed by parameter-free comprehension schemes.
These assert that a certain class of definable sets exist; for example, the
parameter-free $\Sigma^0_1$-comprehension scheme asserts that all sets exist
that are definable by a $\Sigma^0_1$ formula without free set variables. Because
of the existence of universal $\Sigma^0_1$ sets, the parameter-free
$\Sigma^0_1$-comprehension scheme can (in the presence of recursive
comprehension, with parameters) be axiomatized by a single $\Sigma^1_1$ formula.
More generally, parameter-free comprehension schemes are axiomatized by
essentially $\Sigma^1_{n\geq1}$ sentences. They thus appear to be set
existence principles which are not closure conditions, since they do not assert
that (for example) $\PowN$ is closed under the existence of $\Sigma^0_1$
definable sets, but merely that all sets definable by $\Sigma^0_1$ formulas
without set parameters exist.

This presents a challenge to the view defended in this section: if not all set
existence principles are closure conditions, then the claim that I have provided
an analysis of the notion of a set existence principle in terms of closure
conditions is in jeopardy. The obvious response to this worry is that
parameter-free comprehension schemes are not mathematically natural: they are
not equivalent to ordinary mathematical theorems. This is clearly a defeasible
claim---a core mathematical theorem equivalent to parameter-free $\Sigma^0_1$-%
comprehension could be found tomorrow, after all---and moreover, although
parameter-free comprehension schemes are not equivalent to ordinary mathematical
theorems as studied in reverse mathematics, they are often equivalent to
particular instances of those theorems. For instance, the parameter-free
$\Sigma^0_1$-comprehension scheme is equivalent to the restriction of the
Bolzano--Weierstraß theorem to a particular computable, bounded sequence of
real numbers.

This line of argument should be resisted. While in a narrow, technical sense any
derivable sentence is a theorem, in ordinary mathematical practice we do not
grant this appellation so lightly. We have strong antecedent reasons to hold
that there is a difference in degree, if not in kind, between theorems on the
one hand, and mere facts on the other. The Bolzano--Weierstraß theorem is more
informative, more general, more deep, and more useful than its instances,
considered individually or even collectively: unlike them, it is deservedly
called a theorem. I do not intend to precisely articulate the distinction
between theorems and mere mathematical facts---I simply note that it is a
distinction in mathematical practice that, absent compelling evidence to the
contrary, we should treat as a substantial conceptual distinction with attendant
explanatory power. On the basis of this distinction, $\Sigma^0_1$-comprehension
with parameters is mathematically natural, while $\Sigma^0_1$-comprehension
without parameters is not. The restriction in the comprehensiveness constraint
(\ref{se_comprehensiveness}) to mathematically natural principles reflects the
aim of the present paper, namely to determine what kinds of theories can stand
as explications of the concept of a set existence principle as it appears in
reverse mathematics, rather than in mathematics or logic in general. With this
in mind, the fact that parameter-free comprehension schemes are not equivalent to
core mathematical theorems means that they should not be considered to be
counterexamples to the analysis of set existence principles as closure
conditions on $\PowN$.

By striking a balance closer to triviality than noncomprehensiveness, the view
that reversals track closure conditions accommodates the most central part of
reverse mathematics, namely the study of mathematically natural $\Pi^1_2$
theorems. Although such a general account does not, by itself, offer substantial
local explanations of the significance of particular reversals, it does at least
offer a framework within which more fine-grained theorising can be done. The
explanatory power offered by SEC can be partially recovered by acknowledging
that some closure conditions are comprehension schemes, and that comprehension
schemes are a class of principles with distinctive qualities, such that their
necessary use in the proof of an ordinary mathematical theorem will allow
distinctive kinds of explanation. Other classes of principles, such as
separation schemes, may also allow for explanations of the significance of
reversals to any system in that class.

On this version of the standard view, all that can be said in general about the
significance of the equivalences to set existence principles proved in reverse
mathematics is that they show that crucial theorems in diverse areas of ordinary
mathematics require that $\PowN$ satisfies particular closure conditions. These
closure conditions can be captured by natural axioms with an arguably logical
character.
An individual reversal demonstrates the closure condition required to support a
given part of ordinary mathematics, and in some sense picks out an intrinsic
feature of a theorem, namely the resources required to prove it, whether that
be compactness or transfinite recursion. This feature is a proof-invariant
property: every proof of the theorem in question must at some point make use of
this property, although it may appear in different guises.

\section*{Funding}

This work was supported by the Arts and Humanities Research Council [doctoral
studentship, 2011--14]; and the Leverhulme Trust [International Network grant
``Set Theoretic Pluralism''].

\section*{Acknowledgements}

This article is drawn from my doctoral research at the University of Bristol,
and builds builds on work by Walter Dean and Sean Walsh which was unpublished
during the writing of this article, but which has now appeared in part as
\citep{DeaWal2016}. I would like to thank them for their generosity in sharing
their research, and their helpful comments on my own.
I would also like to thank
Marianna Antonutti Marfori,
Michael Detlefsen,
Leon Horsten,
Øystein Linnebo,
Toby Meadows,
Richard Pettigrew,
and Sam Sanders,
for valuable discussions of the topics addressed herein;
the audiences of talks at Helsinki, Leuven, London, and Munich, where I
presented earlier versions of this material;
and the editor and two anonymous referees, whose comments helped sharpen this
paper considerably.

\end{document}